\documentclass[10pt,reqno]{amsart}

\usepackage[T1]{fontenc}
\usepackage[a4paper,asymmetric]{geometry}

\usepackage{latexsym, amsmath, amssymb}

\usepackage{bm}
\usepackage{graphics,epsfig,graphicx,graphics}
\usepackage{color, xcolor}
\usepackage{tikz}
\usetikzlibrary{patterns}
\usepackage[english]{babel}

\usepackage{hyperref}
\newtheorem{thm}{Theorem}[section]
\newtheorem{theorem}{Theorem}[section]

\newtheorem{lem}[thm]{Lemma}

\newcommand{\RR}{\mathbb{R}}

\newcommand{\bd}{\partial}

\newcommand{\diver}{{\rm div}}
\newcommand{\haus}{\mathcal{H}^{N-1}}

\newcommand{\Capa}{{\rm Cap}}

\newcommand{\Om}{\Omega}

\newcommand{\pa}{\partial}

\newcommand{\sdue}[1]{S^2_{ij}(#1)}

\newcommand{\tr}{{\sf{Tr}}}

\begin{document}
\title{A note on an overdetermined problem for the capacitary potential}

\author[C. Bianchini]{Chiara Bianchini}
\author[G. Ciraolo]{Giulio Ciraolo}

\address{C. Bianchini, Dip.to di Matematica e Informatica ``U. Dini'', Universit\`a degli Studi di Firenze, Viale Morgagni 67/A, 50134 Firenze - Italy}
\email{cbianchini@math.unifi.it}

\address{G. Ciraolo,  Dip.to di Matematica e Informatica, Universit\`a di Palermo, Via Archirafi 34, 90123, Palermo - Italy}
\email{giulio.ciraolo@unipa.it}

\maketitle

\begin{abstract}
We consider an overdetermined problem arising in potential theory for the capacitary potential and we prove a radial symmetry result.
\end{abstract}

\noindent {\footnotesize {\bf AMS subject classifications.} 35N25, 35B06, 31B15, 35R25, 52A40.}

\noindent {\footnotesize {\bf Key words.} Overdetermined boundary value problems, electrostatic potential, symmetry, capacity.}


\section{Introduction}
In this note we deal with an overdetermined problem for the electrostatic potential. The electrostatic capacity of a bounded set $\Omega \subset \RR^n$, $n \geq 3$, is defined by 
\begin{equation} \label{capacity_def} 
\Capa (\Omega) = \inf \left\{ \int_{\RR^n} |Dv|^2 dx \: : \ v \in C_c^{\infty} (\RR^n)\,, \ v(x) \geq 1 \quad \forall x \in \Omega \right\} \,, 
\end{equation}
where $C_c^{\infty} (\RR^n)$ denotes the set of $C^\infty$ functions having compact support. It is well-known that it can be equivalently obtained via the asymptotic expansion of the so-called \emph{electrostatic potential} of $\Omega$ (or \emph{capacitary function} of $\Omega$), i.e. 
\begin{equation} \label{capacity_def_asympt}
\Capa(\Omega) = (n-2) \omega_n \lim_{|x| \to \infty} u(x) |x|^{n-2} \,,
\end{equation}
where $\omega_n$ denotes the surface area of the unit sphere in $\RR^n$, and $u$ realizes the minimum of problem \eqref{capacity_def} and hence satisfies
\begin{equation} \label{capacity_EL}
\begin{cases}
\Delta u = 0 & \textmd{in } \RR^n \setminus \overline{\Omega} \,,\\
u=1 & \textmd{on } \partial \Omega \,, \\
\lim_{|x|\to + \infty} u(x) = 0 \,. & 
\end{cases}
\end{equation}
We mention that the electrostatic potential $u$ represents the potential energy of the electrical field induced by the conductor  $\Omega$, normalized so that the voltage difference between $\partial \Omega$ and infinity is one, and hence $\Capa(\Omega)$ represents the total electric charge needed to induce the potential $u$ (see for instance \cite{Ke}). 

A classical question in potential theory is the study of symmetry properties for problem \eqref{capacity_EL}. More precisely, one imposes an extra assumption to Problem \eqref{capacity_EL} and studies how such an overdetermination reflects on the domain $\Omega$. In particular, one can ask whether certain geometric properties of the constraint are inherited by the domain $\Omega$. In this respect, a typical problem is the so-called Serrin's exterior problem, where one assumes that 
\begin{equation} \label{Du_const}
|Du|=c \quad \textmd{on } \partial \Omega \,,
\end{equation}
where $c$ is a positive constant, and one proves that a solution to \eqref{capacity_EL}-\eqref{Du_const} exists if and only if the domain $\Omega$ is a ball. This kind of problem has been successfully solved in \cite{Re} by using the method of moving planes. Other similar problems and related results can be found in \cite{BC,BCS,CFG,Re2,Sa,Se}.

In this note we discuss two kinds of overdeterminations involving the mean curvature $H_{\partial \Omega}$ of $\partial \Omega$ (that is the average of the principal curvatures of $\partial \Omega$). More precisely, we prove the following theorem.

\begin{theorem} \label{thm_main}
Let $\Omega \subset \RR^n$ be a bounded domain with boundary of class $C^2$ and let $u$ be the solution of \eqref{capacity_EL}. If $u$ and $\Omega$ are such that 
\begin{equation} \label{H1}
\int_{\partial \Omega} |Du|^2\left[ H_{\partial \Omega} - \frac{|Du|}{n-2} \right] \, d \mathcal{H}^{n-1} \leq 0,
\end{equation}
or
\begin{equation} \label{H2}
\int_{\bd\Om}|Du|^2\left[ (n-1) H_{\partial \Omega} - \frac{n|Du|}{2(n-2)}\right]  d\haus \leq  \frac{(n-2)^3}{2} \omega_n  \left(\frac{\Capa(\Omega)}{(n-2) \omega_n} \right)^{\frac{n-4}{n-2}}  \,,
\end{equation}
then $\Omega$ is a ball and $u$ is radially symmetric.
\end{theorem}

We mention that in the case that constraint \eqref{H1} holds, Theorem \ref{thm_main} was already proven in \cite{AM}. Indeed, in \cite[Theorem 1.1]{AM} the authors prove the symmetry result by using a conformal reformulation of the problem and by proving the rotational symmetry via a splitting argument. In this respect, we give a different proof of this theorem.

Our approach is very simple and use a chain of integral identities and a basic inequality for symmetric elementary functions (known as Newton's inequality), as in the spirit of \cite{BNST,CS, CRS}.
More precisely, by considering the auxiliary problem for the function 
$$
v=u^{-\frac{2}{n-2}} \,,
$$ 
where $u$ solves \eqref{capacity_EL},
we prove that $v$ must be quadratic, and hence the capacitary function $u$ has radial symmetry. This approach is very flexible and it has been extended to more general settings \cite{BC,BCS}.

It is interesting to notice that from the proof of Theorem \ref{thm_main} (see Step 1 in Section \ref{sect_proof}) we immediately obtain the following lower bound for the capacity, for $n=3$:
\begin{equation} \label{lowerbound}
\Capa (\Omega) \int_{\bd\Om}|Du|^2\left[ (n-1) H_{\partial \Omega} - \frac{n|Du|}{2(n-2)}\right]  d\haus \geq  \frac{(n-2)^3}{2} \omega_n  \,. \end{equation}
This lower bound is optimal, in the sense that the equality sign is attained when $\Omega$ is a ball.

\subsection*{Acknowledgements.} The work has been supported by the FIRB project 2013 ``Geometrical and Qualitative aspects of PDE'' and the GNAMPA of the Istituto Nazionale di Alta Matematica (INdAM).

\section{Preliminaries}
We use the following notation. Let $A = (a_{ij})$ be a $n\times n$ symmetric matrix. We denote by $S_k(A)$, $k\in\{1,\ldots,n\} ,$ the sum of all the principal minors of $A$ of order $k$, so that $S_1(A)=tr(A)$ and $S_n(A)=det(A)$. 
Denoting by 
$$
S^k_{ij}(A) = \frac{\pa }{\pa a_{ij}} S_k(A), 
$$
it holds
$$
S_k(A) = \frac{1}{k} S^k_{ij} (A) a_{ij} ,
$$ 
where here and later the Einstein summation convention is applied.
In particular for $k=2$ 
$$
S^2_{ij}(A) = \frac{\pa }{\pa a_{ij}} S_2(A) = 
\begin{cases} 
-a_{ji} & i \neq j \\
\sum_{k\neq i} a_{kk}  & i=j \,.
\end{cases} 
$$

Notice that $S_k(A)$ are the $k$-th elementary symmetric function of the eigenvalues of $A$; so that
$$
S_k (A) = S_k (\lambda_1,...,\lambda_n) =  \sum_{1\le i_1< ...<i_k\le n} \lambda_{i_1}\cdot ...\cdot \lambda_{i_n},
$$ 
where $\lambda_i$ are the eigenvalues of the matrix $A$.

When $A=D^2v$ we have that 
$$
S_k(D^2v) = \frac{1}{k} \diver (S^k_{ij} (D^2v) v_j) \,,
$$
which follows from the fact that the vector $(S^k_{i1} (D^2v) ,\ldots, S^k_{in} (D^2v) )$ is divergence free for $i=1,\ldots,n$, i.e.
$$
\frac{\pa}{\pa x_{j}} S^k_{ij}(D^2v) = 0,\quad  i=1,\ldots,n .
$$
In particular, for $k=2$ we have
$$
S_2(D^2v) = \frac{1}{2} S^2_{ij}(D^2v) v_{ij} = \frac{1}{2} \diver \Big(S^2_{ij}(D^2v) v_j \Big) \,,
$$
where
$$
S^2_{ij}(D^2v) = \frac{\pa}{\pa v_{ij}} S_2(D^2v) = 
\begin{cases} 
-v_{ji} & i \neq j \\
\Delta v - v_{ii} & i=j \,.
\end{cases} 
$$
Notice that if $L_t=\{v>t\}$ is a super level set of $v$, then 
\begin{equation} \label{pallino}
|Dv|^2 \Delta v = (n-1) H_{\partial L_t} |Dv|^3 + v_i v_{ij}  v_j  \quad \textmd{ on } \partial L_t \,, 
\end{equation}
so that if $\partial L_t$ is oriented such that $\nu= Dv/|Dv|$ then
\begin{equation} \label{pare}
S^2_{ij}( D^2v) v_i v_j = (n-1) H_{\partial L_t}  |Dv|^3  \quad \textmd{ on } \partial L_t \,.
\end{equation}
Two crucial ingredients in the proof of Theorem \ref{thm_main} are contained in next lemmas.
\begin{lem}	[Newton Inequality]\label{lemma-newton}
	Let $A$ be a symmetric matrix in $\RR^{n\times n}$; it holds
	\begin{equation}\label{Newton}
	S_2{(A)}\le\frac{n-1}{2n}\tr(A)^2\, .
	\end{equation}
	Moreover, if ${\tr} (A)\neq 0$ and equality holds in (\ref{Newton}), then
	\begin{equation*}
	A=\frac{{\tr}(A)}{n}\, I \,.
	\end{equation*}
\end{lem}

\begin{lem} \label{lemma_identities}
For any smooth function $v$ and $\gamma \in \mathbb{R}$ we have the following identity:
\begin{multline}\label{identityC}
2 v^{\gamma} S_2(D^2v) = \\ 
=  \diver\Big( \frac{\gamma}{2} v^{\gamma-1}|Dv|^2 Dv + v^{\gamma} S^2_{ij}(D^2v) v_i \Big) - \frac{3}{2} \gamma v^{\gamma-1} |Dv|^2 \Delta v - \frac{\gamma(\gamma-1)}{2} v^{\gamma-2} |Dv|^4   \,.  
\end{multline}
\end{lem}

\begin{proof}
We notice that for $\gamma=0$ \eqref{identityC} is just the definition of $S_2$ and then we may assume $\gamma \neq 0$. 
Identity \eqref{identityC} immediately follows from the following two identities:
\begin{equation}\label{identityA}
 \diver (v^{\gamma} S^2_{ij}(D^2v) v_i )  =2 v^{\gamma}S_2(D^2v) + \gamma v^{\gamma-1} S^2_{ij}(D^2v) v_iv_j  \,,
\end{equation}
and
\begin{equation}\label{identityB}
v^{\gamma-1}S^2_{ij}(D^2v)v_i v_j = \frac32 v^{\gamma-1}|Dv|^2 \Delta v + \frac{\gamma-1}{2} v^{\gamma-2} |Dv|^4 - \frac{1}{2} \diver(v^{\gamma-1} |Dv|^2 Dv ) \,.
\end{equation}
Identity \eqref{identityA} is readily obtained from $ \gamma v^{\gamma-1} v_i= (v^{\gamma})_i$ and 
$$
S_2(D^2v)= \frac{1}{2} S^2_{ij}(D^2v) v_{ij} = \frac12 \diver(S^2_{ij}(D^2v) v_i) \,.
$$
To prove \eqref{identityB} we notice that, since 
$$
S^2_{ij}(D^2v) v_i v_j = |Dv|^2 \Delta v - v_i v_j v_{ij} \,,
$$
we have that
\begin{equation*}
\begin{split}
v^{\gamma-1} S^2_{ij} (D^2v) v_i v_j & = v^{\gamma-1} |Dv|^2 \Delta v - v^{\gamma-1} v_iv_j v_{ij} \\
& = v^{\gamma-1} |Dv|^2 \Delta v + \frac12 \big[ -\diver(v^{\gamma-1}|Dv|^2Dv) + (\gamma-1) v^{\gamma-2}|Dv|^4 + v^{\gamma-1} |Dv|^2 \Delta v \big]\\
& = \frac32 v^{\gamma-1} |Dv|^2 \Delta v + \frac{\gamma-1}{2} v^{\gamma-2} |Dv|^4 - \frac12 \diver(v^{\gamma-1} |Dv|^2 Dv ) \,,
\end{split}
\end{equation*}
which gives \eqref{identityB}. 
\end{proof}

We conclude this section by recalling some well-known properties of the capacitary potential (see \cite{Ke}) which will be useful for the proof of Theorem \ref{thm_main}: 
\begin{equation} \label{u_asymptotic}
\begin{split}
& u= \frac{\Capa(\Omega)}{(n-2) \omega_n} |x|^{2-n} + o(|x|^{2-n}) \,, \\
& u_i= -\frac{\Capa(\Omega)}{\omega_n} |x|^{-n} x_i + o(|x|^{1-n}) \,, \\
& u_{ij}= \frac{\Capa(\Omega)}{\omega_n}  |x|^{-n} \left(n \frac{x_ix_j}{|x|^2} - \delta_{ij} \right) + o(|x|^{-n}) \,,
\end{split}
\end{equation}
as $|x| \to + \infty$.

\section{Proof of Theorem \ref{thm_main}} \label{sect_proof}
\emph{Step 1.} We prove that the reverse inequality holds in \eqref{H1} and \eqref{H2}. More precisely, we shall prove that if $u$ is a solution of \eqref{capacity_EL}, then it satisfies
\begin{equation}\label{BiCi1}
\int_{\bd\Om}|Du|^2\left( H_{\partial \Omega} - \frac 1{n-2}\frac{|Du|}{u}\right) d\haus\ge 0 \,,
\end{equation}
and 
\begin{equation}\label{BiCi2}
\int_{\bd\Om}|Du|^2\left((n-1) H_{\partial \Omega} - \frac{n}{2(n-2)}\frac{|Du|}{u}\right) d\haus\ge \frac{(n-2)^3}{2} \omega_n  \left(\frac{\Capa(\Omega)}{(n-2) \omega_n} \right)^{\frac{n-4}{n-2}}  \,.
\end{equation}

The proof of \eqref{BiCi1} and \eqref{BiCi2} is based on Lemma \ref{lemma_identities} and the Newton Inequality \eqref{Newton} applied to the Hessian matrix of the function $v=u^{-\frac 2{n-2}}$. 
Notice that the function $v$ solves
\begin{equation}\label{pbv}
\begin{cases}
\Delta v=\frac n2 \frac{|Dv|^2}{v}\qquad&\text{ in } \RR^n \setminus\overline{\Om},\\
v=1\qquad&\text{ on }\bd\Om,\\
v\to\infty\qquad&\text{ as } |x| \to + \infty.
\end{cases}
\end{equation}
Moreover, it follows from \eqref{u_asymptotic} that $v$ satisfies
\begin{equation} \label{v_asymptotic}
\begin{split}
& v= \left(\frac{\Capa(\Omega)}{(n-2) \omega_n} \right)^{-\frac{2}{n-2}} |x|^2 + o(|x|^2) \,, \\
& v_i = 2 \left(\frac{\Capa(\Omega)}{(n-2) \omega_n} \right)^{-\frac{2}{n-2}} x_i + o(|x|) \,, \\ 
& v_{ij} =  2\left(\frac{\Capa(\Omega)}{(n-2) \omega_n} \right)^{-\frac{2}{n-2}} \delta_{ij} + o(1) \,, 
\end{split}
\end{equation}
as $|x| \to +\infty$.

We are ready to give the proof of \eqref{BiCi1} and \eqref{BiCi2}. Let $\gamma$ be a fixed parameter to be chosen later and consider \eqref{identityC} applied to the function $v$, solution of \eqref{pbv}. From \eqref{Newton} we have that 
\begin{multline*}
v^{\gamma} \frac{n-1}{n} (\Delta v)^2 \geq \\ 
\geq \diver(v^{\gamma}\sdue{W}v_i) +  \frac{\gamma}{2}\diver(v^{\gamma-1}|Dv|^2 Dv) - \frac{3}{2} \gamma v^{\gamma-1}|Dv|^2\Delta v- \frac\gamma 2(\gamma-1)v^{\gamma-2}|Dv|^4.
\end{multline*}
Since $v$ satisfies \eqref{pbv}, we obtain that 
\begin{equation}\label{divervgammaS2vi}
\diver(v^{\gamma}\sdue{D^2v}v_i)+\frac{\gamma}{2}\diver(v^{\gamma-1}|Dv|^2Dv)\le |Dv|^4v^{\gamma-2}\left(\frac n4(n-1)-\frac \gamma 2(1-\gamma)+\frac 32\gamma\frac n2\right) \,.
\end{equation}
Now, we make our choiche of $\gamma$ so that the right hand side of the above inequality vanishes. This is achieved for $\gamma_1=1-n$ and $\gamma_2=-n/2$. Hence, by choosing $\gamma=\gamma_i$, $i=1,2$, we obtain that $v$ satisfies the following inequality in $\RR^n \setminus \Omega$:
	\begin{equation*}\label{diverle0}
	\diver(v^{\gamma}\sdue{D^2v}v_i)+\frac{\gamma}{2}\diver(v^{\gamma-1}|Dv|^2Dv)\le 0.
	\end{equation*}

Let $R>0$ be such that $\overline{\Omega} \subset B_R$. We integrate the last inequality over $B_R \setminus \Omega$ and apply the divergence theorem: from \eqref{pare} and since $\nu=Dv/|Dv|$ on $\partial \Omega$ we have that
\begin{multline}\label{X}
\int_{\partial \Omega} \left(v^\gamma (n-1) H_{\partial \Omega} |Dv|^2 + \frac{\gamma}{2} v^{\gamma -1} |Dv|^3 \right) d\haus \geq \\
\geq \int_{\partial B_R} \left( v^{\gamma}\sdue{D^2v}v_i\nu_{B_R}^j+\frac{\gamma}{2}v^{\gamma-1}|Dv|^2v_j\nu_{B_R}^j \right) d\haus \,,
\end{multline}
where $\nu_{B_R}$ denotes the outer unit normal vector to $B_R$. 
Now we notice that if $\gamma=\gamma_1$, then \eqref{v_asymptotic} implies that 
\begin{equation}\label{limR=0}
\lim_{R\to\infty}\int_{\bd B_R} v^{\gamma_1}\sdue{D^2v}v_i\nu_{B_R}^j+\frac{\gamma_1}{2}v^{\gamma_1-1}|Dv|^2v_i\nu_{B_R}^i=0,
\end{equation}
while if $\gamma=\gamma_2$ then \eqref{v_asymptotic} yields that 
\begin{equation}\label{limR=capa}
\lim_{R\to\infty}\int_{\bd B_R} v^{\gamma_2}\sdue{D^2v}v_i\nu_{B_R}^j+\frac{\gamma_2}{2}v^{\gamma_2-1}|Dv|^2v_i\nu_{B_R}^i= 2 (n-2) \omega_n  \left(\frac{\Capa(\Omega)}{(n-2) \omega_n} \right)^{\frac{n-4}{n-2}}  \,,
\end{equation}
since $\partial B_R$ is asymptotically a level set of $v$.

By using the fact that $v=1$ on $\Om$ and coupling \eqref{X} and (\ref{limR=0}), we obtain 
	$$
	\int_{\bd\Om} |Dv|^2\left( H_{\partial \Omega} -\frac 12\frac{|Dv|}{v} \right)  \ge 0 \,,
	$$
while from \eqref{X} and \eqref{limR=capa} we find	
$$
	\int_{\bd\Om} |Dv|^2\left((n-1)H_{\partial \Omega} - \frac n4\frac{|Dv|}{v} \right)  \ge 2 (n-2) \omega_n  \left(\frac{\Capa(\Omega)}{(n-2) \omega_n} \right)^{\frac{n-4}{n-2}}  \,.
	$$
By recalling that $v=u^{-\frac 2{n-2}}$, from the last two inequalities we immediately obtain \eqref{BiCi1} and \eqref{BiCi2}.

\emph{Step 2.} From Step 1 we have that the equality sign holds in \eqref{H1} and \eqref{H2}. This means that the equality sign holds in Newton inequality, which implies that for every $x \in \RR^n \setminus \Omega$ there exists a constant $\lambda(x)$ such that
$$
D^2 v(x) = \lambda(x) Id.
$$
It is easy to see that $\lambda$ must be constant. Indeed, let $i=1, \ldots,n$ be fixed and chose any $j \neq i$; we have that
$$
\partial_{x_i} \lambda (x) = \partial_{x_i} u_{x_j x_j} =  \partial_{x_j} u_{x_j x_i} = 0 \,,
$$
which implies that $\lambda$ is constant. Hence, 
\begin{equation} \label{D2v}
D^2 v = c Id \,.
\end{equation} 
From \eqref{pbv} we find that $|Dv|$ is constant on every level surface of $v$. In particular, $|Dv|$ is constant on $\partial \Omega$ and hence from \eqref{pallino} and \eqref{D2v} we find that $H_{\partial\Omega}$ is constant and by using Alexandrov Theorem we conclude that $\Omega$ is a ball. The proof is complete.


\end{document}